\documentclass[12pt]{article}
\usepackage{amssymb,amsmath, amsthm, dsfont}
\usepackage[mathscr]{euscript}
\usepackage{fullpage}
\usepackage{hyperref}
\usepackage{mathtools}
\usepackage{etoolbox}
\usepackage{hyperref}
\usepackage{indentfirst}

\usepackage{enumitem}
\setlist{  
  listparindent=\parindent,
  parsep=0pt,
}

\makeatletter
\newcounter{author}
\renewcommand*\author[1]{%
  \stepcounter{author}%
  \ifnum\c@author=1
  \gdef\@author{#1}%
  \else
  \xdef\@author{\unexpanded\expandafter{\@author\and#1}}%
  \fi
  \csgdef{author@\the\c@author}{#1}}
\newcommand*\email[1]{%
  \csgdef{email@\the\c@author}{#1}}
\newcommand*\address[1]{%
  \csgdef{address@\the\c@author}{#1}}
\AtEndDocument{%
  \xdef\author@count{\the\c@author}%
  \c@author=1
  \print@authors}
\newcommand*\print@authors{%
  \ifnum\c@author>\author@count
  \else
  \print@author{\the\c@author}%
  \advance\c@author by 1
  \expandafter\print@authors
  \fi}
\newcommand*\print@author[1]{%
  \par\medskip
  \begin{tabular}{@{}l@{}}%
    \textsc{\csuse{author@#1}}\\
    \csuse{address@#1}\\
    \textit{E-mail address}:
    \href{mailto:\csuse{email@#1}}{\csuse{email@#1}}
  \end{tabular}}
\makeatother

\newtheorem{theorem}{Theorem}

\newtheorem{remark}{Remark}

\newtheorem{corollary}{Corollary}[theorem]

\newcommand{\supp}[1]{\mathrm{supp}\, #1}
\newcommand{\relint}[1]{\mathrm{relint}(#1)}
\newcommand{\relaint}[1]{\mathrm{core}(#1)}
\newcommand{\affhull}[1]{\mathrm{aff}(#1)}
\newcommand{\co}[1]{\mathrm{co}(#1)}

\title{On barycenters of probability measures}
\date{}
\author{Sergey Berezin}
\address{Aix-Marseille Universit\'e, Centrale Marseille, CNRS, Institut de Math\'ematiques de Marseille,\\ UMR7373, 39 Rue F. Joliot Curie 13453, Marseille, France; St. Petersburg Department of\\ V.A.~Steklov Mathematical Institute of RAS, 27 Fontanka 191023, St. Petersburg, Russia}
\email{servberezin@yandex.ru, sergey.berezin@univ-amu.fr}

\author{Azat Miftakhov}
\address{Moscow State University, Faculty of Mechanics and Mathematics,\\ 1 Leninskiye Gory 119991, Moscow, Russia}
\email{miftakhov-af@tutanota.com}

\begin{document}
\maketitle
\begin{abstract}
  A characterization is presented of barycenters of the Radon probability measures supported on a closed convex subset of a given space. A case of particular interest is studied, where the underlying space is itself the space of finite signed Radon measures on a metric compact and where the corresponding support is the convex set of probability measures. For locally compact spaces, a simple characterization is obtained in terms of the relative interior.
\end{abstract}

\noindent \textbf{1.} The main goal of the present note is to characterize barycenters of the Radon probability measures supported on a closed convex set. Let~$X$ be a Fr\'{e}chet space. Without loss of generality, the topology on~$X$ is generated by the translation-invariant metric~$\rho$ on~$X$ (for details see~\cite{Rudin}).

We denote the set of Radon probability measures on~$X$ by~$\mathscr{P}(X)$. By definition, the barycenter~$a \in X$ of a measure~$\mu \in \mathscr{P}(X)$ is
\begin{equation}
  \label{eq_bary_def1}
  a = \int \limits_X x\, \mu(dx)
\end{equation}
if the latter integral exists in the weak sense, that is, if the following holds
\begin{equation}
  \label{bary_weak}
  \Lambda a = \int \limits_X \Lambda x \, \mu(dx)
\end{equation}
for all~$\Lambda \in X^*$, where~$X^*$ is the topological dual of~$X$. More details on such integrals can be found in~\cite[Chapter~3]{Rudin}.

Note that if~\eqref{eq_bary_def1} exists, then
\begin{equation}
  a = \int \limits_{\supp{\mu}} x\, \mu(dx),
\end{equation}
and, by the Hahn--Banach separation theorem, one has~$a \in \overline{\co{\supp{\mu}}}$, where~$\co{\cdot}$ stands for the convex hull. From now on, we will use the bar over a set in order to denote its topological closure.

Following is the theorem that gives a characterization of barycenters of the measures from~$\mathscr{P}(X)$.
\begin{theorem}
  \label{th1}
  Let~$M \subset X$ be a non-empty compact convex set, and let~$a \in M$. Then, the following statements are equivalent:
  \begin{enumerate}
  \item[\textnormal{(i)}] There exist~$\mu \in \mathscr{P}(X)$ with~$\supp{\mu} = M$ and with barycenter~$a$; \label{th1_one}
  \item[\textnormal{(ii)}] One has
    \begin{equation}
      \label{eq_cond_th1}
      M = \overline{V_a},
    \end{equation}
    where~$V_a = \{x \in M|\, \exists \alpha>0: -\alpha x+(1+\alpha)a \in M \}$.\label{th1_two}
  \end{enumerate}
\end{theorem}

\begin{remark}
  We note that the condition~\eqref{eq_cond_th1} is \textit{non-local} and concerns the whole set~$M$.
\end{remark}
\begin{remark}
  \label{rem2}
  We require~$M$ to be compact in order to ensure the separability of~$M$ and existence of weak integrals (e.g., see~\cite[Theorem~3.27]{Rudin}). If~$X$ is finite-dimensional, the theorem holds without this requirement.
\end{remark}

\begin{proof}[Proof of Theorem~\ref{th1}]
  \begin{enumerate}
  \item[\textnormal{(a)}] First, we prove that~$\mathrm{(i)} \Rightarrow \mathrm{(ii)}$. Let~$c \in M$, and let ~$U_\delta(c)$ be the open ball of radius~$\delta>0$ centered at~$c$. Because~$M$ is the support of~$\mu$, one has~$\mu(U_\delta(c))>0$. Also, since~$M$ is compact, so is~$\overline{U_\delta(c) \cap M}$, and
    \begin{equation}
      c_\delta = \frac{1}{\mu(U_\delta(c))}\int \limits_{U_\delta(c)} x\, \mu(dx) \in M
    \end{equation}
    is well-defined altogether. It is easy to show that
    \begin{equation}
      \lim \limits_{\delta \to +0} c_\delta = c
    \end{equation}
    in the weak topology~$\sigma(X, X^*)$. Indeed, take any~$\Lambda \in X^*$. Since~$\Lambda$ is continuous, for every~$\varepsilon>0$ there exist~$\delta_0>0$ such that~$x \in U_{\delta_0}(c)$ implies~$|\Lambda(x) -\Lambda(c)|=|\Lambda(x-c)|<\varepsilon$. And it follows from the definition of the weak integral that
    \begin{equation}
      \label{eq_lambd_est}
      |\Lambda(c_\delta-c)| \le \frac{1}{\mu(U_\delta(c))}\int \limits_{U_\delta(c)} |\Lambda(x-c)|\, \mu(dx) < \varepsilon
    \end{equation}
    whenever~$\delta \in (0,\delta_0)$. This means that~$c_\delta \to c$ in the weak topology as~$\delta\to +0$.

    Further, for any~$\delta>0$, either~$\mu(U_\delta(c)) = 1$ or~$0<\mu(U_\delta(c))<1$. We show that in both cases~$c_\delta \in V_a$. Indeed, if~$\mu(U_\delta(c)) = 1$, then~$c_\delta = a$ and thus~$c_\delta \in V_a$. If~$0<\mu(U_\delta(c))<1$, set
    \begin{equation}
      \tilde{c}_\delta = \frac{1}{\mu(X \setminus U_\delta(c))}\int \limits_{X \setminus U_\delta(c)} x\, \mu(dx) \in M.
    \end{equation}
    Clearly, $\alpha c_\delta + (1-\alpha) \tilde{c}_\delta \in M$,~$\alpha \in [0,1]$, by convexity. Moreover,
    \begin{equation}
      a = \mu(U_\delta(c)) c_\delta + \big(1-\mu(U_\delta(c))\big) \tilde{c}_\delta.
    \end{equation}
    Therefore, by a simple geometric argument and by the definition of~$V_a$, it is clear that~$c_\delta \in V_a$.

    Since~$X$ is a locally convex space and since~$V_a$ is convex, the closures of~$V_a$ in the weak and original topologies coincide. Consequently, by passing to the limit~$\delta \to +0$, one arrives at
    \begin{equation}
      c = \lim \limits_{\delta \to +0} c_{\delta} \in \overline{V_a}.
    \end{equation}
    This concludes the proof of the claim.

  \item[\textnormal{(b)}] Next, we prove that~$\mathrm{(ii)} \Rightarrow \mathrm{(i)}$ by constructing~$\mu \in \mathscr{P}(X)$ with support~$M$ and barycenter~$a$.

    Being a metric compact, $M$ is separable, hence there exist~$M_0$ such that~$\overline{M_0} = M = \overline{V_a}$. And without loss of generality, one can think that~$M_0=\{x_n\}_{n=1}^\infty \subset V_a$ and~$\overline{\{x_n\}_{n=1}^\infty} = M$. By the definition of~$V_a$, there exist~$\{\alpha_n\}_{n=1}^\infty$ such that~$\alpha_n>0$ and $-\alpha_n x_n+(1+\alpha_n)a \in M$.

    Let us define the discrete measure
    \begin{equation}
      \label{eq_meas_sep}
      \mu = \sum \limits_{k=1}^\infty \frac{1}{2^n} \cdot \frac{\alpha_n \delta_{x_n} + \delta_{-\alpha_n x_n+(1+\alpha_n)a}}{1+\alpha_n},
    \end{equation}
    where~$\delta_x$ is the delta-measure at~$x$. Clearly, this is a Radon probability  measure, and a simple computation shows that its barycenter is~$a$. Indeed, for every~$\Lambda \in X^*$ one has
    \begin{equation}
      \int \limits_{X} \Lambda x\, \mu(dx) = \sum \limits_{k=1}^\infty \frac{1}{2^n} \cdot \frac{\alpha_n \Lambda x_n + (-\alpha_n \Lambda x_n+(1+\alpha_n) \Lambda a)}{1+\alpha_n}= \Lambda a.
    \end{equation}
    
    It remains to prove that~$\supp{\mu}=M$. First, we note that~$\{x_n\}_{n=1}^\infty \subset \supp{\mu}$. Consequently, $M = \overline{\{x_n\}_{n=1}^\infty} \subset \supp{\mu}$, and therefore~$M \subset \supp{\mu}$. By the definition~\eqref{eq_meas_sep}, one also has~$\supp{\mu} \subset M$, which concludes the proof.
    
  \end{enumerate}
  \end{proof}

Further, we will use the following standard notation from convex analysis. For~$a, b \in X$ we define the (line) segment~$[a,b]$ and the open (line) segment~$(a,b)$ by
\begin{equation}
  \begin{aligned}
    &[a,b] = \{x\in X| x = (1-\lambda)a + \lambda b,\, \lambda \in [0,1] \},\\
    &(a,b) = \{x\in X| x = (1-\lambda)a + \lambda b,\, \lambda \in (0,1) \}.
  \end{aligned}
\end{equation}

Let us recall that the \textit{relative interior} of a set~$M$ is
\begin{equation}
  \label{eq_relint_def}
  \relint{M} = \{x \in M|\, \exists U(x): U(x)\cap \affhull{M} \subset M\},
\end{equation}
where~$U(x)$ is an open neighborhood of~$x$ and~$\affhull{M}$ is the affine hull of~$M$. Also, we recall that the \textit{relative algebraic interior} of~$M$ is
\begin{equation}
  \label{eq_relaint_def}
  \relaint{M} = \{x \in M|\, \forall y \in \affhull{M}\ \exists \alpha>0:  [x, -\alpha y+ (1+\alpha)x] \subset M\}.
\end{equation}

It is well-known that any locally compact topological vector space is finite-dimensional (e.g., see~\cite{Rudin}), in which case the following corollary takes place.
\begin{corollary}
  \label{cor1}
  If~$X$ is a locally compact space and $M \subset X$ is a non-empty closed convex set, then the set of barycenters of the Borel probability measures with support~$M$coincides with the relative interior of~$M$.
\end{corollary}
\begin{proof}
  We note that in finite-dimensional spaces any probability Borel measure is Radon. It is also well-known (see~\cite{Zalinescu}) that in such spaces the relative interior and the relative algebraic interior of~$M$ coincide and are non-empty.

  Now, let~$a \in \relint{M} = \relaint{M}$ be any point. By the definition of the relative algebraic interior, for every~$y \in M \subset \affhull{M}$ the segment~$[y,a]$ can be prolonged beyond the point~$a$ within~$M$. This means that~$y \in V_a$, and thus~$M \subset V_a$. Hence, by Theorem~\ref{th1} (see also Remark~\ref{rem2}) there exist a Radon probability measure on~$X$ with~$\supp{\mu}=M$ and barycenter~$a$.

  It is left to prove that if for some~$a \in M$ one has~$\overline{V_a}=M$, then~$a \in \relint{M}$. Notice that~$V_a$ is a non-empty convex set. Since we deal with the finite-dimensional space, $V_a$ has a non-empty relative interior; and~$\relint{V_a} = \relint{\overline{V_a}} = \relint{M}$. Let~$x$ belong to~$\relint{V_a} \subset V_a$. It follows from the definition of~$V_a$ that there exist a segment~$[x,y] \subset M$ such that~$a \in (x,y)$. Since~$x$ also belongs to~$\relint{M}$, there exist an open neighborhood~$U(x)$ of~$x$, satisfying~$U(x)\cap \affhull{M} \subset M$.

  By convexity of~$M$, one obtains
  \begin{equation}
    \label{eq_col1.1_proof_eq1}
    (1-\lambda) (U(x)\cap \affhull{M}) + \lambda y \subset M, \quad \lambda \in [0,1].
  \end{equation}
  It is also easy to verify directly that
  \begin{equation}
    \label{eq_col1.1_proof_eq2}
    (1-\lambda) (U(x)\cap \affhull{M}) + \lambda y = \big((1-\lambda) U(x) + \lambda y\big)\cap \affhull{M}, \quad \lambda \in [0,1].
  \end{equation}
  Combining~\eqref{eq_col1.1_proof_eq1} and \eqref{eq_col1.1_proof_eq2}, and noticing that for~$\lambda \in [0,1)$ the set~$(1-\lambda) U(x) + \lambda y$ is an open neighborhood of the point~$(1-\lambda) x + \lambda y$, one sees that any point of~$(x,y)$ belongs to~$\relint{M}$ by the very definition~\eqref{eq_relint_def}. In particular, this means that~$a \in \relint{M}$.
\end{proof}

\noindent \textbf{2.} It is tempting to think that Corollary~\ref{cor1} takes place in infinite-dimensional spaces as well. Unfortunately, this is not the case even for Hilbert spaces, as the following counterexample shows.

Let~$X$ be the Hilbert space of real sequences endowed with the $l^2$-scalar product, and let~$M$ be the Hilbert cube, a compact convex set,
\begin{equation}
  M=\prod \limits_{k=1}^\infty \left[-\frac{1}{k}, \frac{1}{k}\right].
\end{equation}

We take~$a=\{a_k\}_{k=1}^{\infty} \in M$, where~$a_k=\frac{1}{k+1}$. It is easy to construct a measure~$\mu_k \in \mathscr{P}(\mathbb{R})$ with ~$\supp{\mu_k} = [-1/k, 1/k]$ such that
\begin{equation}
  \int \limits_{[-1/k,1/k]} x\, \mu_k(dx) = \frac{1}{k+1}.
\end{equation}
Having done that, consider the product of measures restricted to~$X$
\begin{equation}
  \mu = \left.\bigotimes_{k=1}^\infty \mu_k\right|_X.
\end{equation}
One usually defines the product of measures on the product of spaces, having in mind the product topology. Even though the corresponding induced topology on~$X$ is strictly coarser than the $l_2$-norm topology, they both generate the same Borel sigma-algebra on~$X$. Thus, it is clear that~$\mu$ can be regarded as a Borel measure on the Hilbert space~$X$; moreover, since~$X$ is a complete and separable metric space, $\mu$ is Radon.

It is clear by construction that~$\supp{\mu} \subset M$. We prove that the other inclusion by reductio ad absurdum. Let~$b \in M$ and suppose that~$\mu(U_{\varepsilon}(b))=0$ for some~$\varepsilon>0$, where~$U_{\varepsilon}(b)$ is the ball of radius~$\varepsilon$ centered at~$b$.

One can choose~$N$ such that
\begin{equation}
  \sum\limits_{n>N} \frac{4}{n^2} < \frac{\varepsilon^2}{2}. 
\end{equation}
Then
\begin{equation}
  \begin{aligned}
    0 = \mu\left\{x \in X \left|\, \sum \limits_{n=1}^{\infty} (x_n - b_n)^2 < \varepsilon^2 \right. \right\} \ge \mu\left\{x \in M \left|\, \sum \limits_{n=1}^{N} (x_n - b_n)^2  < \frac{\varepsilon^2}{2}\right. \right\}\\
    = \bigotimes_{k=1}^N \mu_k\left\{x \in M \left|\, \sum \limits_{n=1}^{N} (x_n - b_n)^2  < \frac{\varepsilon^2}{2}\right. \right\}.
  \end{aligned}
\end{equation}
The latter is clearly positive, which gives a contradiction and yields~$\supp{\mu} = M$. 

Now, we prove that~$a$ is the barycenter of~$\mu$. Thanks to the Riesz representation theorem, there exist~$\{\lambda_k\}_{k=1}^\infty \in X$ such that for every~$x=\{x_k\}_{k=1}^\infty \in X$ one has
\begin{equation}
  \Lambda{x} = \sum_{k=1}^\infty \lambda_k x_k.
\end{equation}
By the definition of the barycenter we write
\begin{equation}
  \int \limits_X \Lambda{x}\, \mu(dx) = \int \limits_M \sum_{k=1}^\infty \lambda_k x_k\, \mu(dx) = \sum_{k=1}^\infty  \lambda_k \int \limits_M  x_k\, \mu(dx) = \sum_{k=1}^\infty  \lambda_k a_k = \Lambda{a},
\end{equation}
where one can interchanged the sum and the integral by dominated convergence since~$M$ is a bounded set in~$X$. This shows that~$a$ is indeed the barycenter of~$\mu$.

Next, we recall that in infinite-dimensional spaces the relative interior and relative algebraic interior do not necessarily coincide (see~\cite{Zalinescu}). However, from~\eqref{eq_relint_def} and~\eqref{eq_relaint_def} one sees that the former is a subset of the latter. Thus, it is sufficient to show that~$a$ does not belong to the relative algebraic interior of~$M$. We prove this claim, again, by contradiction.

Suppose that~$a \in \relaint{M}$. Then the segment~$[0,a]$ can be prolonged beyond the point~$a$ within~$M$. In other words, there exist~$\alpha>0$ such that~$(1+\alpha) a \in M$. The latter is equivalent to
\begin{equation}
  -\frac{1}{k} \le (1+\alpha) a_k \le \frac{1}{k},\quad k=1,2,\ldots. 
\end{equation}
Multiplying by~$k+1$ and letting~$k \to \infty$ yield
\begin{equation}
  -1 \le (1+\alpha) \le 1,
\end{equation}
which contradicts~$\alpha>0$ and concludes the proof.

\noindent \textbf{3.}
Now, we describe the set of barycenters of measures on the space of probability measures. Let~$K$ be a metric compact space and~$X=\mathscr{M}(K)$ the space of signed finite Radon measures on~$K$. By the Riesz--Markov theorem, $X$ can be identified with the topological dual~$C^*(K)$ of the space~$C(K)$ of continuous functions on~$K$. We endow~$C^*(K)$ with the weak-* topology~$\sigma(C^*(K), C(K))$. Having in mind the canonical embedding~$C(K) \hookrightarrow C^{**}(K)$, one can say that this topology is the weakest topology which makes continuous all the functionals from~$C^{**}(K)$ that correspond to elements of~$C(K)$. This topology is locally convex as is the corresponding topology~$\tau_w$ on~$X$. The restriction of~$\tau_w$ to the convex set~$M=\mathscr{P}(K) \subset X$ of probability measures on~$K$ produces the usual topology of weak convergence on~$M$ and thus makes this set compact.

The barycenter~$\mu \in X$ of a measure~$\eta \in \mathscr{P}(X)$ is, by definition,
\begin{equation}
  \mu = \int \limits_X \nu\, \eta(d\nu)
\end{equation}
if the latter integral exists in the weak sense. That is, since~$(C^*(K))' = C(K)$, where~$(\cdot)'$ is the topological dual in the weak-* topology, $\mu$ is the barycenter of~$\eta$ if and only if for every~$f \in C(K)$,
\begin{equation}
  \label{eq_3_mu}
  \int \limits_{K}f(x) \mu(dx) = \int \limits_{X} \left(\int \limits_{K} f(x)\nu(dx)\right)\, \eta(d\nu).
\end{equation}

Also, note that
\begin{equation}
  \mu = \int \limits_{\supp{\eta}} \nu\, \eta(d\nu),  
\end{equation}
and, by the Hahn--Banach separation theorem, one has~$\mu \in \overline{\co{\supp{\eta}}}$.

The following result characterizes measures from~$X$ with support~$M$.
\begin{theorem}
  The set of barycenters of the measures from~$\mathscr{P}(X)$ with support~$M$ coincides with the set of the measures from~$M$ with support~$K$.
\end{theorem}
\begin{proof}
  \begin{enumerate}
  \item[\normalfont{(a)}] First, we prove that the barycenter of a measure from~$\mathscr{P}(X)$ with support~$M$ is a measure from~$M$ with support~$K$.
    
    Take any~$\eta \in \mathscr{P}(X)$ such that~$\supp{\eta} = M$, and let~$\mu \in M$ be its barycenter. We prove that~$\supp{\mu}=K$ by contradiction.

    Indeed, suppose this is not the case. Then, there exist a non-zero nonnegative continuous bounded function~$f \in C_b(K)$ such that
    \begin{equation}
      \label{eq_3_eq1}
      \int \limits_K f(x) \mu(dx)=0.
    \end{equation}
    Using~\eqref{eq_3_mu} one gets
    \begin{equation}
      \int\limits_M \int \limits_{K} f(x) \nu(dx) \eta(d \nu) = 0,
    \end{equation}
    and since the integrand is non-negative,  
    \begin{equation}
      \label{eq_3_eq2}
      \int \limits_{K} f(x) \nu(dx) = 0
    \end{equation}
    $\eta$-almost surely on~$M$.

    The latter, in fact, holds for all~$\nu \in M = \mathscr{P}(K)$ due to continuity in~$\nu$ of the left-hand side of~\eqref{eq_3_eq2} with respect to the topology of weak convergence.

    Consequently, by choosing~$\nu$ to be the delta measure at an arbitrary point in~$K$, one immediately obtains 
    \begin{equation}
      f(x)=0, \quad x \in K.
    \end{equation}
    This contradicts the fact that~$f$ is non-zero and concludes the proof of the claim. 
  \item[\normalfont{(b)}]
    Now, assume that~$\mu \in M$ and~$\supp{\mu}=K$. Let
    \begin{equation}
      A=\left\{(a_1,a_2,\ldots) \in [0,1]^\infty\left|\, a_j\ge 0, \sum_{j=1}^{\infty} a_j=1\right.\right\}
    \end{equation}
    be a closed subset of~$[0,1]^\infty$ endowed with the~$l_1$-norm. Since~$A$ is separable, there exist a Radon probability measure~$\lambda$ on~$[0,1]^\infty$ with support~$A$ (e.g., see the proof of Theorem~\ref{th1}).

    Let us also introduce the Radon probability measure~$\lambda \otimes \mu^\infty = \lambda \otimes \bigotimes_{j=1}^\infty \mu_j$ on $A\times K^\infty = A \times \prod_{j=1}^\infty K_j$, where the~$\mu_j$ are copies of~$\mu$ and the~$K_j$ are copies of~$K$. It is easy to see that
    \begin{equation}
      \label{eq_supp_lambda_muinf}
      \supp{\lambda \otimes \mu^\infty} = A \times K^{\infty}.
    \end{equation}
    Indeed, for any open neighborhood~$U(c)$ of~$c=(c_a;c_1,\ldots) \in A \times K^{\infty}$, by the definition of the product topology, there exist an open set of the form
    \begin{equation}
      U_a(c_a) \times \prod_{j=1}^\infty U_j(c_j), 
    \end{equation}
    where~$U_a(c_a) \subset A$ and~$U_j(c_j) \subset K_j$ are open neighborhoods of~$c_a$ and~$c_j$, respectively, such that~$U_j(c_j) \ne K_j$ only for finitely many~$j \in \mathbb{N}$. Then, for large enough~$N$ one has
    \begin{equation}
      (\lambda \otimes \mu^\infty)(U(c)) \ge \lambda(U_a(c_a)) \prod_{j=1}^N  \mu(U_j(c_j))>0,
    \end{equation}
    which proves~\eqref{eq_supp_lambda_muinf}.

    The next step is to define the map~$F: A \times K^{\infty} \to M$ by
    \begin{equation}
      F(a, x) = \sum \limits_{j=1}^{\infty} a_j \delta_{x_j}.
    \end{equation}
    It is easy to show that~$F$ is continuous. Indeed, let~$a^{(n)} \to a^* \in A$ in~$l_1$-norm and~$x^{(n)} \to x^* \in K^\infty$ in the product topology. We will prove that~$F(a^{(n)},x^{(n)})$ converges to~$F(a^{*},x^{*})$ weakly. For every~$f \in C(K)$,
    \begin{equation}
      \begin{aligned}
        \left|\int \limits_Kf(y)\, F(a^{(n)},x^{(n)})(dy) - \int \limits_Kf(y) F(a^{*},x^{*})(dy)\right| \le \sum \limits_{j=1}^{\infty}|a_j^{(n)}f(x_j^{(n)}) - a_j^{*}f(x_j^{*})|\\
        \le \sup_{x\in K}{|f(x)|}\, \|a^{(n)} - a^{*}\|_{l_1} + \sum \limits_{j=1}^{\infty} a_j^{*}|f(x_j^{(n)}) - f(x_j^{*})| \to 0,
      \end{aligned}
    \end{equation}
    where the latter term tends to zero thanks to the dominated convergence theorem. This proves the continuity of~$F$.
    
    Now, let us define the measure~$\eta$ to be the pushforward of~$\lambda \otimes \mu^{\infty}$ under~$F$:
    \begin{equation}
      \label{eq_prod_meas}
      \eta = (\lambda \otimes \mu^{\infty}) \circ F^{-1},
    \end{equation}
    which is readily verified to be a Radon probability measure.

    We prove that this measure is supported on~$M$. Indeed, since it is known (e.g., see~\cite[Ex. 8.1.6]{Bogachev}) that
    \begin{equation}
      \overline{F(A \times K^{\infty})} = M,
    \end{equation}
    for every open neighborhood~$U(\nu)$ of~$\nu \in M$ there exist~$(a,x) \in A \times K^\infty$ such that~$F(a,x) \in U(\nu)$. Consequently, due to~$F$ being continuous and due to~\eqref{eq_supp_lambda_muinf}, one has~$\eta(U(\nu))>0$, and thus, since~$\nu$ is arbitrary,~$\supp{\eta} = M$.

    It is left to check that the barycenter of~$\eta$ is~$\mu$. One can write
    \begin{equation}
      \label{eq_t2_proof_baryc}
      \begin{aligned}
        &\int\limits_M \int \limits_{K} f(y) \nu(dy) \eta(d \nu) = \int \limits_{A \times K^{\infty}} \int\limits_{K} f(y)\, F(a,x)(dy)\, (\lambda \otimes \mu^\infty)(da,dx)\\
        &= \sum \limits_{j=1}^{\infty} \int \limits_{A} a_j \lambda(da) \int\limits_{K^{\infty}} f(x_j)  \mu^{\infty}(dx)= \sum \limits_{j=1}^{\infty} \int \limits_{A} a_j \lambda(da) \int\limits_{K} f(x)  \mu(dx) =\int\limits_K f(x) \mu(dx),
      \end{aligned}
    \end{equation}
    where we use the definition~\eqref{eq_prod_meas} of~$\eta$, Fubini's theorem, and the dominated convergence to interchange the sum and the integrals.

    According to~\eqref{eq_3_mu}, the formula~\eqref{eq_t2_proof_baryc} means exactly that the barycenter of~$\eta$ is~$\mu$. This concludes the proof of the theorem. 
  \end{enumerate}
\end{proof}

As the final remark we point out that our proof relies heavily on the fact that~$K$ is compact; even though, barycenters are well-defined for a wider class of Radon probability measures (with finite first moments). An open question of interest is to characterize such measures as well.

\bigbreak

\noindent \textbf{Acknowledgments.} We would like to thank V.~Bogachev for bringing the problem considered in this note to our attention and A.~Borichev for helpful discussions and valuable comments. Also, we greatly appreciate the detailed responses of the reviewers. Their remarks and comments, without a doubt, helped us improve our paper.

S.B. is supported by the European Research Council (ERC) under the European Union’s Horizon 2020 research and innovation programme, grant  647133 (ICHAOS). A.M. is supported by the RFBR grants 14-01-90406, 14-01-00237 and the SFB 701 at Bielefeld University.

\newpage
\end{document}